\documentclass[a4paper,11pt,leqno]{article}
\usepackage{geometry}
\usepackage{tikz,tikz-cd,amsmath,amsfonts,amssymb,graphicx}
\usetikzlibrary{matrix,arrows,decorations.pathmorphing}
\usepackage[latin1]{inputenc}
\usepackage{amsthm}
\usepackage[autostyle]{csquotes} 
\usepackage{enumerate}
\usepackage[mathscr]{euscript}
\usepackage{mathtools}
\usepackage{hyperref}
\usepackage{url}
\usepackage{scalerel}
\usepackage{dsfont}
\usepackage{stackengine,wasysym}

\usepackage[noabbrev,nameinlink,capitalize]{cleveref} 

\numberwithin{equation}{section}
\theoremstyle{definition}

\newtheorem{theorem}[equation]{Theorem}
\newtheorem{corollary}[equation]{Corollary}
\newtheorem{definition}[equation]{Definition}
\newtheorem{lemma}[equation]{Lemma}
\newtheorem{example}[equation]{Example}
\newtheorem{proposition}[equation]{Proposition}
\newtheorem{remark}[equation]{Remark}

\makeatletter
\newcommand{\colim@}[2]{%
	\vtop{\m@th\ialign{##\cr
			\hfil$#1\Operator@font colim$\hfil\cr
			\noalign{\nointerlineskip\kern1.5\ex@}#2\cr
			\noalign{\nointerlineskip\kern-\ex@}\cr}}%
}
\newcommand{\colim}{%
	\mathop{\mathpalette\colim@{\rightarrowfill@\scriptscriptstyle}}\nmlimits@
}
\renewcommand{\varprojlim}{%
	\mathop{\mathpalette\varlim@{\leftarrowfill@\scriptscriptstyle}}\nmlimits@
}
\renewcommand{\varinjlim}{%
	\mathop{\mathpalette\varlim@{\rightarrowfill@\scriptscriptstyle}}\nmlimits@
}
\makeatother
\newcommand{\cate}[1]{\mathscr{#1}}

\newcommand{\Hom}[3]{\text{Hom}_{\scalebox{1}{$\scriptscriptstyle #1$}}(#2, #3)}

\newcommand{\Cl}[1]{\text{Cl}({#1})}

\newcommand{\pf}[1]{#1_{\ast}}

\newcommand{\Str}[1]{\text{Str}_{{#1}}}

\DeclareMathOperator{\Fun}{Fun}

\DeclareMathOperator{\opposite}{op}

\DeclareMathOperator{\exitpath}{Exit}
\DeclareMathOperator{\singularset}{Sing}

\newcommand{\op}{^{\opposite}}
\newcommand{\C}{\cate{C}}

\newcommand{\Cat}{\cate{C}\mathrm{at}}
\newcommand{\rnum}{\mathbb{R}}
\newcommand{\exit}[2]{\exitpath_{{#1}}({#2})}

\newcommand{\sing}[1]{\singularset({#1})}

\newcommand{\holink}[2]{{#1}[{#2}]}

\newcommand{\subd}[1]{\text{Sd}({#1})}
\newcommand{\Dec}[1]{\text{D\'ec}_{{#1}}}

\providecommand{\infcat}{\Cat_{\scalebox{1}{$\scriptscriptstyle \infty$}}}
\providecommand{\spaces}{\cate{S}}

\newsavebox{\pullback}
\sbox\pullback{%
	\begin{tikzpicture}%
		\draw (0,0) -- (1ex,0ex);%
		\draw (1ex,0ex) -- (1ex,1ex);%
\end{tikzpicture}}
\newsavebox{\pushout}
\sbox\pushout{%
	\begin{tikzpicture}%
		\draw (0ex,0ex) -- (0ex,1ex);%
		\draw (0ex,1ex) -- (1ex,1ex);%
\end{tikzpicture}}
\tikzset{%
	symbol/.style={%
		draw=none,
		every to/.append style={%
			edge node={node [sloped, allow upside down, auto=false]{$#1$}}}
	}
}

{}

\begin{document}

\title{Finiteness and finite domination in stratified homotopy theory}
  
	\author{Marco Volpe\footnote{University of Toronto, 27 King's College Cir, Toronto, ON M5S 1A1, Canada. email:marco.volpe@utoronto.ca.}}

	\maketitle

\begin{abstract}
    In this paper, we study compactness and finiteness of an $\infty$-category $\C$ equipped with a conservative functor to a finite poset $P$. We provide sufficient conditions for $\C$ to be compact in terms of strata and homotopy links of $\C\rightarrow P$. Analogous conditions for $\C$ to be finite are also given. From these, we deduce that, if $X\rightarrow P$ is a conically stratified space with the property that the weak homotopy type of its strata, and of strata of its local links, are compact (respectively finite) $\infty$-groupoids, then $\exit{P}{X}$ is compact (respectively finite). This gives a positive answer to a question of Porta and Teyssier. If $X\rightarrow P$ is equipped with a conically smooth structure (e.g. a Whitney stratification), we show that $\exit{P}{X}$ is finite if and only the weak homotopy types of the strata of $X\rightarrow P$ are finite. The aforementioned characterization relies on the finiteness of $\exit{P}{X}$, when $X\rightarrow P$ is compact and conically smooth. We conclude our paper by showing that the analogous statement does not hold in the topological category. More explicitly, we provide an example of a compact $C^0$-stratified space whose exit paths $\infty$-category is compact, but not finite. This stratified space was constructed by Quinn. We also observe that this provides a non-trivial example of a $C^0$-stratified space which does not admit any conically smooth structure.
\end{abstract} 
\tableofcontents

\section{Introduction}

A topological space $X$ is said to be \textit{(homotopically) finite} if it is homotopy equivalent to a finite CW-complex, and is said to be \textit{finitely dominated} if it is a retract up to homotopy of a finite CW-complex. A standard reference analyzing the difference between these two notions is \cite{10.2307/1970382}. From a modern perspective, both notions can be understood purely in terms of $\spaces$, the $\infty$-category of $\infty$-groupoids. Specifically, when $X$ is nice enough, we have that $X$ is finite if and only if its weak homotopy type belongs to the smallest subcategory of $\spaces$ closed under pushouts, and containing the empty and the contractible $\infty$-groupoid. Similarly, $X$ is finitely dominated if and only if its weak homotopy type is a compact object in $\spaces$. 

In this paper, we seek to study analogues of finiteness and finite domination within the setting of \textit{stratified topological spaces}. To formulate these analogues in the stratified setting, we find it easier to generalize their higher categorical reformulations mentioned above. Therefore, for the reader's convenience, we begin this introduction by briefly summarizing some of the recent foundational work on defining an appropriate $\infty$-category encompassing the homotopy theory of stratified spaces. 

Recall that any poset $P$ can be supplied with a topology, called the \textit{Alexandroff topology}, with open subsets consisting of the upward closed subsets of $P$. For us, a stratified space is then a topological space $X$ equipped with a function to a poset $P$, which is continuous with respect to the Alexandroff topology. A nice family of stratified spaces, which contains most of the examples coming from algebraic and differential geometry, is that of \textit{conically stratified spaces}. Roughly speaking, these are stratified spaces that, locally around each stratum, admit mapping cylinder neighbourhoods. To any conically stratified space $X\rightarrow P$, \cite{lurie2017higher} associates an $\infty$-category $\exit{P}{X}$. $\exit{P}{X}$ comes equipped with a conservative functor to the poset $P$, and should be thought of as a refinement of the weak homotopy type of $X$. The only non-invertible morphisms in $\exit{P}{X}$ can be visualized as those paths which start in a deeper stratum and immediately leave for the next one. With this in mind, \cite{haine2018homotopy} proposed the following as an effective definition of the homotopy theory of stratified spaces. 

\begin{definition}
    Let $P$ be any poset. We define a \textit{$P$-layered $\infty$-category} to be an $\infty$-category $\C$ equipped with a conservative functor $\C\rightarrow P$. We denote by $\Str{P}$ the full subcategory of ${\infcat}_{/P}$ spanned by the $P$-layered $\infty$-categories. $P$-layered $\infty$-categories are also sometimes referred to as \textit{$P$-stratified weak homotopy types}.
\end{definition}

In \cite[Theorem 0.1.1]{haine2018homotopy}, Haine proves a stratified version of the homotopy hypothesis. In particular, this result shows that any $P$-layered $\infty$-category is of the form $\exit{P}{X}$, for some stratified space $X\rightarrow P$ (see also \cite{waas2024presenting}). We regard Haine's result as convincing evidence that one is entitled to think of the $\infty$-category of $\Str{P}$ as the correct homotopy theory of stratified spaces. 

Just as simplices are the basic building blocks of spaces up to weak homotopy equivalences, the basic building blocks in stratified homotopy theory are some sort of $P$-directed simplices. More precisely, for any strictly ascending chain of elements $p_1<\dots< p_n$ of the poset $P$, there is a $P$-layered $\infty$-category $\{p_1 <\dots< p_n\}\rightarrow P$, given by the corresponding inclusion of posets. The objects $\{p_1 <\dots< p_n\}\rightarrow P$ with $n\leq 2$ can be shown to generate $\Str{P}$ under colimits. Intuitively, this reflects the idea that categories can be built by subsequently attaching objects and arrows. With these observations at hand, we may come back to the formulation of the sought definitions of finiteness and finite domination in stratified homotopy theory.

\begin{definition}\label{finite/compactP-lay}
    A $P$-layered $\infty$-category $\C\rightarrow P$ is said to be \textit{finite} if it belongs to the smallest full subcategory of $\Str{P}$ closed under pushouts and containing the empty $\infty$-category and all objects of the form $\{p_1<\dots< p_n\}\rightarrow P$, with $n\leq 2$. $\C\rightarrow P$ is said to be \textit{compact}, or \textit{finitely dominated}, if it is a retract in $\Str{P}$ of a finite $P$-layered $\infty$-category.
\end{definition}

\cref{finite/compactP-lay} is very natural, but too abstract to be directly checked in practice. The first contribution of this paper is to provide reasonably easy to check conditions that imply finiteness or compactness of a $P$-layered $\infty$-category, when $P$ is a finite poset. These conditions are formulated in terms of  \textit{homotopy links}. Homotopy links were originally introduced by Quinn (see \cite{quinn1988homotopically}) as a purely homotopy theoretical tool to deal with stratified spaces in absence of mapping cylinder neighbourhoods around strata. We recall the purely $\infty$-categorical definition of homotopy links.

\begin{definition}
    Let $\C\rightarrow P$ be any $P$-layered $\infty$-category, and let $p_1<\dots< p_n$ be any ascending chain of elements in $P$. We define the \textit{$p_1<\dots< p_n$-homotopy link} of $\C\rightarrow P$ to be the $\infty$-groupoid
    $$\C[p_1<\dots< p_n]\coloneqq\Hom{\Str{P}}{\{p_1<\dots< p_n\}}{\C}.$$
    When $n=1$, we denote $\C[p_1]$ by $\C_{p_1}$. Moreover, $\C_{p_1}$ is called the \textit{$p_1$-stratum} of $\C\rightarrow P$.
\end{definition}

For any pair $p < q$, there is an evaluation map $\C[p < q]\rightarrow \C_p$. This map should be thought of as analogous to the projection from the $q$-stratum of a mapping cylinder neighbourhood of the $p$-stratum. The precise relationship is explained in \cref{holinksandlocallinks}. Our criterion for finiteness and compactness in terms of homotopy links is the following.

\begin{proposition}[\cref{cptstra+cptlocallinksiscpt}]
    Let $P$ be any finite poset, and let $\C\rightarrow P$ be a $P$-layered $\infty$-category. Assume that the following two conditions hold for $\C$. 
    \begin{enumerate}
        \item For all $p\in P$, the $p$-stratum $\C_p$ is a compact (respectively finite) $\infty$-groupoid.
        \item For all pairs $p<q$, the fibers of the map $\C[p<q]\rightarrow\C_p$ are compact (respectively finite) $\infty$-groupoids.
    \end{enumerate}
    Then $\C\rightarrow P$ is compact (respectively finite).
\end{proposition}

As an application, we provide a positive answer to a question of Porta and Teyssier (see \cite[Conjecture 7.10]{porta2022topological}) in the following theorem.

\begin{theorem}[\cref{conjecturemauro}]\label{conjecturemauroinintro}
    Let $P$ be any finite poset, and let $X\rightarrow P$ be a conically stratified space. Suppose that $X$ satisfies the following properties.
        \begin{enumerate}[(i)]
            \item For all $p\in P$, $\sing{X_p}$ is a compact (respectively finite) $\infty$-groupoid.
            \item For all $p\in P$, $x\in X_p$ and $q>p$, $\sing{Z_q}$ is a compact (respectively finite) $\infty$-groupoid. Here $Z\rightarrow P_{>p}$ is a stratified space appearing in any conical chart $U\times C(Z)\hookrightarrow X$ of $X$ centered at $x$.
        \end{enumerate}
        Then $\exit{P}{X}$ is a compact (respectively finite) $\infty$-category.
\end{theorem}

Let us note that our criteria for compactness of $P$-layered $\infty$-categories might be useful in the study of the generalized character variety $\mathbf{Cons}_P(X)$ introduced in \cite[Section 7]{porta2022topological}. In particural, as a consequence of \cite[Proposition 7.7]{porta2022topological}, we deduce that for any stratified space $X\rightarrow P$ satisfying the assumptions of \cref{conjecturemauroinintro}, the algebraic stack $\mathbf{Cons}_P(X)$ is locally geometric.

As a corollary of \cref{conjecturemauroinintro}, when the stratified space $X\rightarrow P$ is equipped with more geometric structure,  we find that the compactness (respectively finiteness) of $\exit{P}{X}$ can be determined solely on strata.

\begin{corollary}[\cref{cpt/finstrata=>cprfinexitinC0/smooth}]\label{corollarycharonstrataC0consmoothinintro}
    Let $P$ be a finite poset, and let $X\rightarrow P$ be a $C^0$-stratified space. Then $\exit{P}{X}$ is compact if and only if, for all $p\in P$, $\sing{X_p}$ is compact. Moreover, when $X$ is equipped with a conically smooth structure, we have that $\exit{P}{X}$ is finite if and only if, for all $p\in P$, $\sing{X_p}$ is finite.
\end{corollary}

We refer to \cite{ayala2017local} for the definitions of conically smooth structures and $C^0$-stratifications. Readers unfamiliar with these precise definitions may safely think of the former as Whitney stratifications (see \cite{nocera2023whitney}), and the latter as a topological analogue (e.g. pseudomanifolds). 

We point out that, in the presence of a conically smooth structure on a compact stratified space $X\rightarrow P$, we are able to obtain more refined information about the finiteness of $\exit{P}{X}$ purely in terms of strata, as opposed to compactness in the $C^0$ case. This is because, when $X\rightarrow P$ is compact and conically smooth,  $\exit{P}{X}$ can be proven to be finite (see \cite[Proposition 2.19]{volpe2022verdier}). On the other hand, we can only establish the compactness of $\exit{P}{X}$ when $X\rightarrow P$ is compact and $C^0$-stratified (see \cref{cptC0strat=>sratacpt}). 

In the last part of the paper, we investigate whether finiteness of $\exit{P}{X}$ holds for a compact $C^0$-stratified space $X\rightarrow P$. Building upon the results in \cite{ho2024complements}, we are able to deduce finiteness in the case of the $C^0$-stratification of a compact topological manifold given by picking a closed, locally flat submanifold (see \cref{finiteexitlocallyflat}). However, in general, the answer is negative, as demonstrated in the work of Quinn (see \cite[Proposition 2.1.4]{quinn1982ends}). We also observe that the lack of finiteness of $\exit{P}{X}$ may serve as a useful criterion to determine the non-existence of conically smooth structures on a compact stratified space $X\rightarrow P$.

\begin{theorem}[\cref{examplequinn}]
      There exists a compact $C^0$-stratified space $X\rightarrow P$ such that the $\infty$-category $\exit{P}{X}$ is not finite. Moreover, $X\rightarrow P$ does not admit any conically smooth structure compatible with its $P$ stratification.
\end{theorem}

\subsection{Acknowledgements}

We thank Ko Aoki, David Ayala, Peter Haine, Sander Kupers, Mauro Porta, Jean-Baptiste Teyssier and Shmuel Weinberger for conversations regarding the topic treated in this paper.
	
	\section{Abstract criterion for compactness and finiteness}

This section is devoted to proving an abstract criterion for compactness (finiteness) of a $P$-layered $\infty$-category, when $P$ is a finite poset. We start by collecting a few preliminary well-known results on finite $\infty$-categories that will be used in what follows. 

 \subsection{Preliminaries}

 \begin{definition}
     An $\infty$-category $\C$ is said to be \textit{finite} if it belongs to the smallest full subcategory of $\infcat$ which contains $\emptyset$, $[0]$, and $[1]$ and is closed under pushouts.
     
     An $\infty$-category $\C$ is said to be \textit{strongly finite} if it is finite and, for each $x, y\in \C$, the space $\Hom{\C}{x}{y}$ is finite. 
 \end{definition}

     \begin{example}
    	Any finite poset $P$ is a strongly finite $\infty$-category. Beware that there exist finite $\infty$-categories which are not strongly finite. The standard example is the homotopy type of $S^1$.
    \end{example}

 \begin{lemma}
     Let $\C\rightarrow P$ be a $P$-layered $\infty$-category. Then the following are equivalent
     \begin{enumerate}
         \item $\C\rightarrow P$ is compact (respectively finite) as an object in $\Str{P}$;
         \item $\C$ is compact (respectively finite) as an object in $\infcat$.
     \end{enumerate} 
 \end{lemma}

 \begin{proof}
        The part of the lemma concerning compactness is proven in \cite[Lemma A.3.10]{haine2024exodromy}, so we focus on finiteness.

        We first observe that the functor $\Str{P}\rightarrow\infcat$ preserves finite objects. Indeed, by \cite[Observation A.3.10]{haine2024exodromy}, the inclusion $\Str{P}\hookrightarrow\infcat{_{/P}}$ preserves colimits. Since the functor $\infcat{_{/P}}\rightarrow\infcat$ preserves colimits, we deduce that $\Str{P}\rightarrow\infcat$ does too. The claim about preservation of finite objects then follows by observing that objects of the form $\{p_1<\cdots<p_n\}\rightarrow P$, with $n\leq 2$, are sent to $\infty$-categories equivalent to either $[0]$ or $[1]$.
    	
    	Assume now that $\C$ is finite as an object of $\infcat$. Consider the composite 
    	$${\infcat}_{/\C}\rightarrow {\infcat}_{/P}\xrightarrow{\text{Env}_P}\Str{P},$$
        where $\text{Env}_P$ is the left adjoint to the inclusion $\Str{P}\hookrightarrow{\infcat}_{/P}$ (see for example \cite[Observation A.3.3]{haine2024exodromy}). Both functors preserve colimits, and the composite sends objects of the form $[n]\rightarrow \C$ to conservative functors $[m]\rightarrow P$, with $m\leq n$. Indeed, this holds because any localization of $[n]$ is equivalent to $[m]$, for some $m\leq n$. The desired conclusion then follows by observing that $\C$, as the terminal object of ${\infcat}_{/\C}$, belongs to the smallest full subcategory of ${\infcat}_{/\C}$ closed under finite colimits and containing objects of the type $[n]\rightarrow \C$.
 \end{proof}

    \begin{definition}
    	Let $I$ be a strongly finite $\infty$-category, and let $F\in\Fun(I^{op}, \spaces)$. We say that $F$ is \textit{finite} if it belongs to the smallest full subcategory of $\Fun(I^{op}, \spaces)$ closed under pushouts which contains the initial object and representable presheaves. 
    \end{definition}

    \begin{lemma}\label{finiteobjstroinglyfinitecat}
    	Let $I$ be a strongly finite $\infty$-category, and let $F\in\Fun(I^{op}, \spaces)$. Then $F$ is compact (respectively finite) if it takes values in $\spaces^{\omega}$ (respectively $\spaces^{fin}$).
    \end{lemma}	

    \begin{proof}
    	The statement about compact objects follows from \cite[Proposition 2.8]{aoki2023tensor}. Now let $i$ be any object of $I$, and consider the evaluation functor $\text{ev}_i : \Fun(I^{op},\spaces)\rightarrow\spaces$. 
    	Since $I$ is strongly finite, we know that $\text{ev}_i$ sends representable presheaves to finite spaces. Thus, since $\text{ev}_i$ preserves pushouts, we see that it has to preserve finite objects. 
    	
    	Now assume that $F\in\Fun(I^{op},\spaces)$ takes values in finite spaces. Consider the left fibration $p : I_{/F}\rightarrow I$ given by the category of elements of $F$. Since $I$ is finite and by assumption the fibers of $p$ are also finite, by \cite[Remark 6.5.4]{calmes2023hermitian} we see that $I_{/F}$ must be finite. Therefore, writing $F$ as colimit of representable functors indexed by $I_{/F}$, we get that $F$ is finite, and so the proof is concluded.
    \end{proof}

    \subsection{Criterion for compactness and finiteness}

    In this section we give a first criterion for compactness (respectively finiteness) of a $P$-layered $\infty$-category in terms of its higher homotopy links. This criterion is demonstrated making use of an alternative model for the homotopy theory of stratified spaces, the so called \textit{$P$-d\'ecollages}. This alternative model is built upon the equivalence between $\infcat$ and \textit{complete Segal spaces}. We refer to \cite{haine2018homotopy} for more details about $P$-d\'ecollages. 

    \begin{definition}
         Let $P$ be any poset. We define the \textit{poset of subdivisions of $P$}, denoted by $\subd{P}$, to be the full subcategory of $\Str{P}$ spanned by objects of the form $\{p_1<\cdots<p_n\}\rightarrow P$, for any strictly ascending chain of elements $p_1<\cdots<p_n$ in $P$.
    \end{definition}

 \begin{definition}\label{decollages}
    A presheaf $F\in\Fun(\subd{P}\op,\spaces)$ is said to a \textit{$P$-d\'ecollage} if, for any $\{p_1<\cdots<p_n\}\in\subd{P}$, the canonical map
     $$F(\{p_1<\cdots<p_n\})\rightarrow F(\{p_1<p_2\})\times_{\{p_2\}}\cdots\times_{\{p_{n-1}\}}F(\{p_{n-1}<p_n\})$$
     is an isomorphism. We denote by $\Dec{P}$ the full subcategory of $\Fun(\subd{P}\op, \spaces)$ spanned by the $P$-d\'ecollages. 

     The inclusion $\Dec{P}\hookrightarrow\Fun(\subd{P}\op,\spaces)$ preserves limits and filtered colimits, and therefore it admits a left adjoint, that we denote by $\text{Seg}_P : \Fun(\subd{P}\op, \spaces)\rightarrow\Dec{P}$.
 \end{definition}

 As a Bousfield localization of $\infcat{_{/P}}$, $\Str{P}$ has all colimits. It follows that the inclusion $\subd{P}\hookrightarrow\Str{P}$ extends uniquely to a colimit preserving functor $\Fun(\subd{P}\op, \spaces)\rightarrow\Str{P}$. Moreover, the fact that we isomophisms $$\{p_1<p_2\}\cup_{\{p_2\}}\cdots\cup_{\{p_{n-1}\}}\{p_{n-1}<p_n\}\simeq\{p_1<\cdots<p_n\}$$ in $\Str{P}$, implies that $\subd{P}\hookrightarrow\Str{P}$ induces a colimit preserving functor 
 \begin{align}\label{comparedecandsd}
     \Dec{P}\rightarrow\Str{P}.
 \end{align}
 \cite[Theorem 1.1.7]{haine2018homotopy} proves the following result.

 \begin{theorem}\label{equivdecandsd}
     The functor (\ref{comparedecandsd}) is an equivalence of $\infty$-categories.
 \end{theorem}

 \begin{remark}\label{finpreshtofinitestrP}
    	Let $F\in\Fun(\subd{P}\op, \spaces)$ be a presheaf lying in the smallest subcategory of $\Fun(\subd{P}\op, \spaces)$ closed under pushouts and containing the initial object and all representable presheaves. Denote by $L_P$ the composition of $\text{Seg}_P$ with the equivalence (\ref{comparedecandsd}). Then $L_P(F)$ is finite. This follows immediately by observing that each representable presheaf on $\subd{P}$ lies in $\Dec{P}$, and $\text{Seg}_P$ preserves colimits. 
    \end{remark}

    \begin{proposition}\label{charfin/cpt}
    	Let $P$ be a finite poset, and let $\C\rightarrow P$ be a $P$-layered $\infty$-category. Suppose that, for each sequence $\{p_1<\dots < p_n\}$, $\holink{\C}{p_1<\dots < p_n}\in\spaces$ is compact (respectively finite). Then $\C\rightarrow P$ is compact (respectively finite) as an object in $\text{Str}_P$.
    \end{proposition}

    \begin{proof}
    	By \cref{equivdecandsd}, there exists $F\in\Fun(\subd{P}\op, \spaces)$ such that $L_P(F)\simeq(\C\rightarrow P)$. Moreover, the assumptions on $\C\rightarrow P$ imply that we can assume that $F$ takes values in compact (respectively finite) spaces. Note that, when $P$ is finite, $\text{Sd}(P)$ is a finite poset as well. Therefore, by \cref{finiteobjstroinglyfinitecat}, it follows that $F$ is a finite object in $\Fun(\text{Sd}(P)^{op}, \spaces)$. But $(\C\rightarrow P)\simeq L_P(F)$, and so by \cref{finpreshtofinitestrP} we see that $\C\rightarrow P$ is compact (respectively finite). 
    \end{proof}

 We'll also need the following partial converse to \cref{charfin/cpt}.

 \begin{proposition}\label{restrstratapeter}
     Let $\C\rightarrow P$ be $P$-layered $\infty$-category, and assume that $\C\rightarrow P$ is compact (finite). Then, for any $p\in P$, $\C_p$ is compact (finite).
 \end{proposition}

 \begin{proof}
     This is proven in \cite[Proposition A.3.17]{haine2024exodromy}.
 \end{proof}

     \begin{definition}
     	Let $P$ be any poset, and let $\C\rightarrow P$ be in $\text{Str}_P$. We say that $\C\rightarrow P$ has \textit{compact local links} if each fiber of the map $\holink{\C}{p< q}\rightarrow \C_p$ is a compact object in $\spaces$. Similarly, we say that $\C\rightarrow P$ has \textit{finite local links} if each fiber of the map $\holink{\C}{p< q}\rightarrow \C_p$ is a finite object in $\spaces$.
     \end{definition}

      \begin{remark}
    	Not all $P$-layered $\infty$-categories have finite or compact local links. In fact, one can find finite ones that do not have finite or compact local links. An example is given by taking $P=[1]$ and $S^1\leftarrow\ast\rightarrow S^1$, considered as an object in $\Fun(\subd{P}\op,\spaces) \simeq \Str{P}$. 
    \end{remark}

     \begin{lemma}\label{cptloclink=>cptrestrtostrata}
         Let $P$ be any poset, and let $\C\rightarrow P$ be in $\text{Str}_P$. Assume that $\C\rightarrow P$ has compact local links. Let $q$ be any element in $P$, and denote  by $j:\C_q\rightarrow \C$ the inclusion of the $q$-stratum. Then the right Kan extension functor 
         $$\pf{j}:\Fun(\C_q, \spaces)\rightarrow\Fun(\C, \spaces)$$
         preserves filtered colimits. 
     \end{lemma}

     \begin{proof}
         Let $x$ be any object in $\C$, and let $\C_p$ be the unique stratum to which $x$ belongs. By the pointwise formula for right Kan extensions, to show that $\pf{j}$ preserves filtered colimits it suffices to prove that the slice $(\C_q)_{x/}$ is a compact object in $\infcat$. 

         If $p$ is not less or equal to $q$, the slice $(\C_q)_{x/}$ is empty, and when $p=q$, it has an initial object. Therefore, we can assume that $p< q$. We show that $(\C_q)_{x/}$ fits in a pullback square 
         $$
         \begin{tikzcd}
             (\C_q)_{x/}\arrow[r]\arrow[d] & \holink{\C}{p< q}\arrow[d] \\
             {[0]}\arrow[r, "x"] & \C_p.
         \end{tikzcd}
         $$
         Since we assumed that $\C$ has compact local links, this will conclude the proof.

         By definition of the slice, we have a pullback square
         $$
         \begin{tikzcd}
             (\C_q)_{x/}\arrow[r]\arrow[d]\arrow[dr, phantom, "\usebox\pullback" , very near start, color=black] & \Fun([1],\C) \arrow[d, "{(\text{ev}_0,\text{ev}_1)}"] \\
             {[0]}\times \C_q\arrow[r, "x\times j"] & \C\times \C.
         \end{tikzcd}
         $$
         But the square above can be factored as 
         $$
         \begin{tikzcd}
             (\C_q)_{x/}\arrow[r]\arrow[d] & \holink{\C}{p< q}\arrow[r]\arrow[d]\arrow[d, "{(\text{ev}_0,\text{ev}_1)}"]\arrow[dr, phantom, "\usebox\pullback" , very near start, color=black] & \Fun([1], \C) \arrow[d, "{(\text{ev}_0,\text{ev}_1)}"] \\
             {[0]}\times \C_q\arrow[r] & \C_p\times \C_q\arrow[r] & \C\times \C,
         \end{tikzcd}
         $$
         where the square on the right is a pullback by definition. Therefore, also the square on left is a pullback. By composing it with the pullback square
         \begin{equation*}
             \begin{tikzcd}
             {[0]}\times \C_q\arrow[r]\arrow[d]\arrow[dr, phantom, "\usebox\pullback" , very near start, color=black] & \C_p\times \C_q\arrow[d] \\
             {[0]}\arrow[r, "x"] & \C_p
            \end{tikzcd}
        \end{equation*}
        we then get the desired conclusion.
    \end{proof}
    
    \begin{proposition}\label{cptstra+cptlocallinksiscpt}
    	Let $P$ be any finite poset, and let $\C\rightarrow P$ be in $\Str{P}$. Assume that $\C\rightarrow P$ has compact (finite) local links, and that the strata of $\C\rightarrow P$ are compact (finite). Then $\C\rightarrow P$ is compact (finite).
    \end{proposition}
    
    \begin{proof}
    	By \cref{charfin/cpt}, it will suffice to show that our assumptions imply that for each sequence $\{p_1<\dots < p_n\}$, $\holink{\C}{p_1<\dots < p_n}\in\spaces$ is compact (finite). We have a pullback square
        $$
          \begin{tikzcd}
\holink{\C}{p_1<\dots < p_n} \arrow[d] \arrow[r] \arrow[dr, phantom, "\usebox\pullback" , very near start, color=black] & \holink{\C}{p_{n-1}< p_n} \arrow[d] \\
\holink{\C}{p_1<\dots < p_{n-1}} \arrow[r]       &  {\C}_{p_{n-1}}.                  
           \end{tikzcd}
        $$
        By assumption, we know that the fibers of the right vertical map are compact (finite). Therefore, it will suffice to show that $\holink{\C}{p_1<\dots < p_{n-1}}$ is compact (finite). Considering a finite number of pullback squares as above, we can reduce the question to proving that $\holink{\C}{p_1 < p_2}$ is compact (finite). But we know that $\C_{p_1}$ is compact (finite), and the fibers of $\holink{\C}{p_1< p_2}\rightarrow \C_{p_1}$ are compact (finite) as well. Therefore, we can conclude that $\holink{\C}{p_1< p_2}$ is compact (finite).
    \end{proof}
	
\section{Applications to conically stratified spaces}

In this section, we apply the abstract criterion for compactness and finiteness proven previously to $P$-layered $\infty$-categories coming from topology. We find conditions on a conically stratified topological space $X\rightarrow P$ that imply that its exit paths $\infty$-category $\exit{P}{X}$ is compact or finite. We then obtain more refined results, when specializing to the case when $X\rightarrow P$ is $C^0$-stratified, or admits a conically smooth atlas. We refer to \cite[Appendix A]{lurie2017higher} for a definition of $\exit{P}{X}$, and to \cite{ayala2017local} for the notion of $C^0$-stratifications and conically smooth atlases.

\subsection{Geometric interpretation of local links}

We start by providing a geometric interpretation of the condition of having compact or finite local links. We study the fibers of the map 
$$\exit{P}{X}[p<q]\rightarrow\sing{X_p}$$
when $X\rightarrow P$ is a conically stratified space, and relate those to the geometric local links of $X$. 

 Let $Z\rightarrow Q$ be any conically stratified topological space. The continuous map
	
	$$
	\begin{tikzcd}[row sep= tiny]
		Z\times I \arrow[r] & C(Z)  \\
		(z, t)\arrow[r, mapsto] & {[z, t]}
	\end{tikzcd}
	$$
	
	induces a functor
	
	$$\exit{Q}{Z}\rightarrow\Fun(\exit{[1]}{I}, \exit{C(Q)}{C(Z)})\simeq\Fun([1], \exit{C(Q)}{C(Z)})$$
	
	where the stratification $I\rightarrow[1]$ is given by taking $\{0\}\subset I$ as a closed stratum. Since all paths given by the map above start at the cone point $\ast$, we get an induced map
	$$\exit{Q}{Z}\rightarrow\exit{C(Q)}{C(Z)}_{\ast/}$$
	and thus, by adjunction a functor
	\begin{equation}\label{joinmapstocone}
		[0]\ast\exit{Q}{Z}\rightarrow\exit{C(Q)}{C(Z)},
	\end{equation}

    which sends the initial object on the left hand side to the cone point.
    
    \begin{lemma}\label{exitcone}
    	The map (\ref{joinmapstocone}) is an equivalence of $\infty$-categories.
    \end{lemma}

    \begin{proof}
    	The functor (\ref{joinmapstocone}) is essentially surjective. The cone point is in the essential image, and since any point $[z, t]\in \rnum_{>0}\times Z$ is connected to $[z, 1]$ by a path which doesn't leave the stratum of $[z, t]$, we get the claim. Thus, it remains to prove that (\ref{joinmapstocone}) is fully faithful. Arguing as above, one sees that it is sufficient to prove that, for any $z\in Z$, the map 
    	$$\ast\simeq\Hom{[0]\ast\exit{Q}{Z}}{0}{z}\rightarrow A\coloneqq\Hom{\exit{C(Q)}{C(Z)}}{\ast}{[z, 1]}$$
    	in an isomorphism in $\spaces$. We now prove this by providing an homotopy between the identity of $A$ and the constant map with value the path $s\mapsto[z, s]$. For any exit path $\gamma : I \rightarrow C(Z)$ starting at the cone point and ending at $[z, 1]$ and any $s\in I$, we denote $[\gamma^Z(s), \gamma^I(s)]\coloneqq \gamma(s)$. One checks that the homotopy 
    	$$
    	\begin{tikzcd}[row sep= tiny]
    		I\times A \arrow[r] & A \\
    		(t, \gamma)\arrow[r, mapsto] & s\mapsto{[\gamma^Z(1-(1-t)(1-s)), st + (1-t)\gamma^I(s)]}
    	\end{tikzcd}
    	$$
    	does the job.
    \end{proof}

	Let $s: X\rightarrow P$ be a conically stratified space, $p\in P$ and let $x\in X$ be any point such that $s(x) = p$. Suppose that $U\times C(Z)$ is a conical chart centered in $x$. Then, for any $q\in P$ such that $p\leq q$, one gets a map 
	$$
		\begin{tikzcd}[row sep= tiny]
			\sing{Z_q} \arrow[r, "c_{x}"] &  \holink{\exit{P_{\geq p}}{U\times C(Z)}}{p\leq q} \\
			z \arrow[r, mapsto] & t\mapsto (x, [z, t])
		\end{tikzcd}
	$$
	and commutative squares
	
	\begin{equation}\label{htpyfibev0}
		\begin{tikzcd}
			\sing{Z_q} \arrow[r, "{c_{x}}"] \arrow[d] &  \holink{\exit{P_{\geq p}}{U\times C(Z)}}{p\leq q} \arrow[r] \arrow[d, "{\text{ev}_0}"] & \holink{\exit{P}{X}}{p\leq q} \arrow[d, "{\text{ev}_0}"] \\
			{[0]} \arrow[r, "x"] & \sing{U} \arrow[r] & \sing{X_p}.
		\end{tikzcd}
    \end{equation}

	\begin{proposition}\label{holinksandlocallinks}
		The outer rectangle in (\ref{htpyfibev0}) is a pullback in $\spaces$.
	\end{proposition}

    \begin{proof}
    	Since the inclusion $\spaces\hookrightarrow\infcat$ preserves limits, it will suffice to prove that (\ref{htpyfibev0}) is a pullback in $\infcat$. By \cite[Proposition A.7.9]{lurie2017higher}, the right rectangle is a pullback, and so it suffices to prove that the left one is a pullback. Notice that $$\holink{\exit{P_{\geq p}}{U\times C(Z)}}{p\leq q}\simeq \sing{U}\times \holink{\exit{P_{\geq p}}{C(Z)}}{p\leq q}$$ and that the cospan 
    	$$[0]\xrightarrow{x}\sing{U}\xleftarrow{\text{ev}_0}\holink{\exit{P_{\geq p}}{U\times C(Z)}}{p\leq q}$$
    	is obtained as a product of the two cospans
    	$$[0]\xrightarrow{x}\sing{U}\xleftarrow{\text{id}}\sing{U}, [0]\xrightarrow{\text{id}}[0]\leftarrow\holink{\exit{P_{\geq p}}{C(Z)}}{p\leq q}.$$
    	Thus, it suffices to prove that the map 
    	$$\sing{Z_q}\xrightarrow{c_{x}}\holink{\exit{P_{\geq p}}{C(Z)}}{p\leq q}$$
    	is an equivalence, where $x\in C(Z)$ now denotes the cone point. The map $c_x$ is induced by the commutativity of the outer rectangle in the diagram
    	$$
    	\begin{tikzcd}
    		\sing{Z_q} \arrow[r] \arrow[d] & \exit{P_{\geq p}}{C(Z)} \arrow[r, "\simeq"] \arrow[d, "s"] & \exit{P_{\geq p}}{C(Z)}_{x/} \arrow[r] \arrow[d, "s"] & \text{Fun}([1], \exit{P_{\geq p}}{C(Z)}) \arrow[d, "s"] \\
    		{[0]} \arrow[r, "q"] & N(P_{\geq p}) \arrow[r, "\simeq"] & N(P_{\geq p})_{p/}\arrow[r] & \text{Fun}([1], N(P_{\geq p})),
    	\end{tikzcd}
    	$$ 
    	and thus it suffices to prove that the outer rectangle is a pullback in $\infcat$. The left triangle is evidently a pullback, the middle one is a pullback because by \cref{exitcone} $x$ and $p$ are initial objects of $\exit{P_{\geq p}}{C(Z)}$ and $N(P_{\geq p})$ respectively, and thus the horizontal maps are equivalences. We are then only left to show that the right rectangle is a pullback. Consider the following commutative diagram
    	$$
    \begin{tikzcd}
    	& {\exit{P_{\geq p}}{C(Z)}_{x/}} \arrow[rr] \arrow[dd] \arrow[ld] &                                                           & {\Fun([1], \exit{P_{\geq p}}{C(Z)})} \arrow[dd, "\text{ev}_0"] \arrow[ld] \\
    	N(P_{\geq p})_{p/} \arrow[rr] \arrow[dd] &                                                                 & {\Fun([1], N(P_{\geq p}))} \arrow[dd, "\text{ev}_0"] &                                                                                \\
    	& {[0]} \arrow[rr, "x"] \arrow[ld]                             &                                                           & {\exit{P_{\geq p}}{C(Z)}} \arrow[ld]                                           \\
    	{[0]} \arrow[rr, "p"]                 &                                                                 & N(P_{\geq p}).                                            &                                                                               
    \end{tikzcd}
    	$$
    We want to show that the upper horizontal square is a pullback. The front vertical square is a pullback, and so it suffices to show that the composition of the upper horizontal square and the front vertical square is a pullback. But this coincides with the composition of the back vertical square and the lower horizontal square, which are both evidently pullbacks, and so we may conclude.
    \end{proof}

    The following theorem gives a proof to \cite[Conjecture 7.10]{porta2022topological}.

    \begin{theorem}\label{conjecturemauro}
        Let $P$ be any finite poset, and let $X\rightarrow P$ be a conically stratified space. Suppose that $X$ satisfies the following properties.
        \begin{enumerate}[(i)]
            \item For all $p\in P$, $\sing{X_p}$ is a compact (respectively finite) $\infty$-groupoid.
            \item For all $p\in P$, $x\in X_p$ and $q>p$, $\sing{Z_q}$ is a compact (respectively finite) $\infty$-groupoid. Here $Z\rightarrow P_{>p}$ is a stratified space appearing in any conical chart $U\times C(Z)\hookrightarrow X$ of $X$ centered at $x$.
        \end{enumerate}
        Then $\exit{P}{X}$ is a compact (respectively finite) $\infty$-category.
    \end{theorem}

    \begin{proof}
        By \cref{cptstra+cptlocallinksiscpt}, it suffices to show that $\exit{P}{X}$ has compact (finite) local links and strata. After \cref{holinksandlocallinks}, one sees that these two conditions are equivalent to (i) and (ii) in the statement of the theorem.
    \end{proof}

\subsection{Applications}

    \begin{definition}
        Let $\C$ be any $\infty$-category. We say that $\C$ is \textit{cofinally compact} if the limit functor 
        $$\Fun(\C, \spaces)\xrightarrow{\text{lim}}\spaces$$
        preserves filtered colimits. 
    \end{definition}

    \begin{remark}\label{cpthtpytypeeasychar}
        It's not difficult to show that any compact $\infty$-category is cofinally compact. Moreover, if $\C$ is an $\infty$-groupoid, one can show that $\C$ is compact if and only if it is cofinally compact. However, there are many $\infty$-categories which are cofinally compact, but not compact. For example, take $I$ to be any infinite set, and let $\C$ be the category obtained by adding an initial object to $I$. Since the limit functor on $\C$ is given by evaluating on the initial objects, we see that $\C$ is cofinally compact. However, $\C$ is not a compact $\infty$-category, because $I$ is not compact.
    \end{remark}

    \begin{lemma}\label{cpthausex=>limpreservesfiltcolim}
        Let $X\rightarrow P$ be an conically stratified space. Assume that $X$ is compact Hausdorff, and locally of singular shape. Then  $\exit{P}{X}$ is cofinally compact.
    \end{lemma}

    \begin{proof}
        Since $X$ is compact and Hausdorff, by \cite[Corollary 7.3.4.12]{lurie2009higher} we know that the global section functor preserves filtered colimits. The lemma is then proven by observing that, via the exodromy equivalence (see \cite[Theorem A.9.3]{lurie2017higher}), the global section functor corresponds to taking the limit indexed by $\exit{P}{X}$.
    \end{proof}

    \begin{corollary}\label{cpthausex=>cptstrata}
        Let $X\rightarrow P$ be a conically stratified space, and assume that $X$ is compact and Hausdorff. Moreover, assume that $\exit{P}{X}\rightarrow P$ has compact local links. Then, for each $p\in P$, $\sing{X_p}$ is a compact object in $\spaces$.
    \end{corollary}

    \begin{proof}
        By \cref{cpthtpytypeeasychar}, $\sing{X_p}$ is a compact object in $\spaces$ if and only the limit functor preserves filtered colimits. If $j:\sing{X_p}\rightarrow\exit{P}{X}$ is the inclusion, then we can factor the limit functor on $\sing{X_p}$ as
        $$\Fun(\sing{X_p},\spaces)\xrightarrow[]{\pf{j}}\Fun(\exit{P}{X},\spaces)\xrightarrow[]{\text{lim}}\spaces.$$
        But by \cref{cptloclink=>cptrestrtostrata} and \cref{cpthausex=>limpreservesfiltcolim} both functors preserve filtered colimits, and therefore we may conclude.
    \end{proof}

    \begin{theorem}\label{cptC0strat=>sratacpt}
    	Let $X\rightarrow P$ be a compact $C^0$-stratified topological space. Then $\exit{P}{X}$ is a compact object in $\infcat$. Moreover, when $X$ is equipped with a conically smooth structure (e.g. a Whitney stratified space) $\exit{P}{X}$ is finite.
    \end{theorem}

    \begin{proof}
        By \cite[Theorem A.9.3]{lurie2017higher}, we know that $X\rightarrow P$ is exodromic. Therefore it makes sense to consider $\exit{P}{X}$.
    	Notice that the compactness of $X$ forces $P$ to be finite (see \cite[Lemma 3.1]{volpe2022verdier}). We proceed by induction on the cardinality of $P$. 
     
        Denote by $d$ the cardinality of $P$. If $d=1$, then $X$ is a compact topological manifold. Using for example \cite{west1977mapping}, we know that its homotopy type is finite, and therefore compact. 

        Assume that $d>1$. By \cite[Lemma 2.2.2]{ayala2017local}, we know that $X$ has a basis given by open subsets which are isomorphic as stratified spaces to $\rnum^n\times C(Z)$, where $Z$ is a compact $C^0$-stratified space whose stratified poset $Q$ has cardinality smaller than $d$. Therefore, by the inductive assumption $\exit{Q}{Z}$ is compact, and hence all its strata are compact by \cref{restrstratapeter}. By \cref{holinksandlocallinks}, we get that $\exit{P}{X}$ has compact local links. Thus, by \cref{cptstra+cptlocallinksiscpt} it suffices to show that the strata of $X$ have compact homotopy type. But this follows immediately from \cref{cpthausex=>cptstrata}.

        Let us now assume that $X$ is equipped with a conically smooth atlas. In this situation, finiteness of $\exit{P}{X}$ has already been proven in \cite[Proposition 2.19]{volpe2022verdier}. Here we provide a more direct argument, relying on the results of the previous section. Observe that, arguing as in the $C^0$ case, it suffices to prove that any open stratum $U$ in $X$ has the homotopy type of a finite CW-complex. By resolution of singularities (see \cite[Proposition 7.3.10]{ayala2017local}), we know that $U$ is the interior of a compact smooth manifold with corners $Y$. A routine application of the existence of collarings of corners shows that $U$ is homotopy equivalent to $Y$. Therefore, one deduce finiteness of $U$ from the finiteness of $Y$. 
    \end{proof}

    We deduce the following corollary.

    \begin{corollary}\label{cpt/finstrata=>cprfinexitinC0/smooth}
        Let $P$ be a finite poset, and let $X\rightarrow P$ be a $C^0$-stratified space. Then $\exit{P}{X}$ is compact if and only if, for all $p\in P$, $\sing{X_p}$ is compact. Moreover, when $X$ is equipped with a conically smooth structure, we have that $\exit{P}{X}$ is finite if and only if, for all $p\in P$, $\sing{X_p}$ is finite.
    \end{corollary}

    \begin{proof}
        By \cref{restrstratapeter}, we know that compactness (finiteness) of $\exit{P}{X}$ implies compactness (finiteness) of $\sing{X_p}$ for all $p\in P$. Therefore, assume that a $C^0$-stratified space with the property that $\sing{X_p}$ is compact for all $p\in P$. We want to show that $\exit{P}{X}$ is compact. By \cref{conjecturemauro}, it suffices to prove that $\exit{P}{X}$ has compact local links. But this follows immediately from \cref{cptC0strat=>sratacpt}, since $X$ has conical charts of the form $\rnum^n\times C(Z)$, where $Z\rightarrow Q$ is a compact $C^0$-stratified space. The last assertion of the corollary is deduced analogously, by observing that the presence of a conically smooth atlas provides conical charts as above where $Z$ is itself equipped with a conically smooth atlas.
    \end{proof}

    \section{The case of compact $C^0$-stratifications}

    Let $X\rightarrow P$ be a compact $C^0$-stratified space. It is natural to wonder whether the finiteness of $\exit{P}{X}$ may be obtained regardless of the presence of a conically smooth atlas for $X$. Our argument for finiteness of $\exit{P}{X}$ in \cref{cptC0strat=>sratacpt} in the case of $X$ compact and conically smooth relies on the existence of certain blow-ups. Such blow-ups should not be expected to exist when working in the topological category (see \cite{kupers2020there}). Therefore, there is no a priori reason why one should expect to have finiteness of $\exit{P}{X}$ when a conically smooth atlas is not given to $X$. In the rest of this section, we provide an example of a compact $C^0$-stratified space $X\rightarrow P$ with the property that $\exit{P}{X}$ is compact but not finite. The space we consider is due to Quinn (see \cite{quinn1982ends}). We deduce that this stratified space does not admit any conically smooth structure compatible with its stratification.

    We start by considering a special class of $C^0$-stratified spaces, not necessarily smoothable, whose exit path $\infty$-category is nevertheless finite.

    \begin{proposition}\label{finiteexitlocallyflat}
        Let $X$ be a compact topological manifold, and let $Y$ be $k$-dimensional closed locally flat submanifold of $X$. Let $[1]$ be the poset with two elements $\{0<1\}$. Consider the stratification $X\rightarrow [1]$, whose initial stratum is given by $Y$. Then $X\rightarrow [1]$ is $C^0$-stratified, and $\exit{[1]}{X}$ is finite.
    \end{proposition}

    \begin{proof}
        By the assumption of local flatness of $Y$, around each point of $Y$ one may find a euclidean chart of $X$ of the form $\rnum^k\times \rnum^l\hookrightarrow X$, where $\rnum^k\times\{0\}$ is mapped onto an open subset of $Y$. Therefore, one obtains an open embedding $$\rnum^k\times C(S^{l-1})\cong\rnum^k\times \rnum^l\hookrightarrow X$$ which respects the $[1]$ stratification. Hence we deduce that $X\rightarrow[1]$ is $C^0$-stratified.

        We want to show that $\exit{[1]}{X}$ is finite. The $0$-stratum $X_0 = Y$ is a compact topological manifold, and therefore homotopy equivalent to a finite CW-complex (see for example \cite{west1977mapping}). Moreover, the conical charts provided above show that the local links around the points in $X_0$ are spheres. Hence, by \cref{conjecturemauro}, we are only left to show that that $X_1 = X\setminus Y$ has the homotopy type of a finite CW-complex. This is proven in \cite{ho2024complements}.
    \end{proof}

    \subsection{The orbit type stratification}

The goal of this subsection is to recall the definition of the \textit{orbit-type stratification} on spaces equipped with a continuous $G$-action. We specifically define it in terms of a continuous map to the poset of closed subgroups of $G$, equipped with the Alexandroff topology.

Let $G$ be any topological group, and let $X$ be any topological space equipped with a continuous $G$ action. Recall that, for any $x\in X$, the \textit{isotropy group} of $x$, denoted by $G_x$, is defined to be the subgroup of $G$ given by the elements that fix the point $x$. Equivalently, $G_x$ is the fiber at the point $(x, x)\in X\times X$ of the continuous function
$$
\begin{tikzcd}[row sep = tiny]
    G\times X\arrow[r] & X\times X \\
    (g, x)\arrow[r, mapsto] & (gx, x).
\end{tikzcd}
$$
Therefore, if $X$ is $T_1$, one can define a function of sets
\begin{equation}\label{isotropymap}
   \begin{tikzcd}[row sep = tiny]
    X^{\delta}\arrow[r, "\text{iso}"] & \Cl{G} \\
    x\arrow[r, mapsto] & G_x.
   \end{tikzcd} 
\end{equation}
Above $X^{\delta}$ denotes the underlying set of $X$ equipped with the discrete topology, and $\Cl{G}$ denotes the set of closed subgroups of $G$.

\begin{lemma}
    Suppose that $X$ is Hausdorff. Then the map (\ref{isotropymap}) promotes to a continuous function
    $$X\xrightarrow{\text{iso}}\Cl{G}\op.$$
    Here $\Cl{G}\op$ denotes the set $\Cl{G}$ equipped with the Alexandroff topology associated to the ordering given by reverse inclusion.
\end{lemma}

\begin{proof}
    Let $H$ be any closed subgroup of $G$. Unraveling the definition of the Alexandroff topology, one sees that proving the lemma amounts to showing that 
    $$X_{\leq H}\coloneqq\{x\in X\mid G_x\supseteq H\}$$
    is a closed subset of $X$. For a fixed $g$, the set 
    $$\text{eq}(g, e) \coloneqq \{x\in X\mid gx = x\}$$
    is the preimage of the diagonal under the continuous map
    $$
    \begin{tikzcd}[row sep = tiny]
        X\arrow[r] & X\times X \\
        x\arrow[r, mapsto] & (gx,  x).
    \end{tikzcd}
    $$
    Since $X$ is Hausdorff, we deduce that $\text{eq}(g, e)$ is a closed subset of $X$. The proof is concluded by observing that we have the equality
    $$\{x\in X\mid G_x\supseteq H\} = \bigcap_{g\in H}\text{eq}(g, e).$$
    \end{proof}

\begin{definition}
    Let $G$ be any topological group, and let $X$ be any Hausdorff topological space equipped with a continuous $G$-action. We define the \textit{orbit type stratification}  of $X$ to be the continuous map $X\xrightarrow{\text{iso}}\Cl{G}\op$. For each $H\in\Cl{G}$, the corresponding stratum will be denoted by $X_H$. 
\end{definition}

The set $\Cl{G}$ admits a natural $G^{\delta}$-action, given by conjugation. Since conjugation respects inclusions, the set of orbits $\Cl{G}/G^{\delta}$ has induced ordering, and the quotient map $\Cl{G}\rightarrow\Cl{G}/G^{\delta}$ is order preserving. We denote by $\Cl{G}\op/G$ the topological space whose underlying set is $\Cl{G}/G^{\delta}$, equipped with the Alexandroff topology associated with the ordering induces by the reverse inclusion. By the functoriality of the Alexandroff topology, the quotient map $\Cl{G}\op\rightarrow\Cl{G}\op/G$ is continuous.

The composition
$$X\xrightarrow{\text{iso}}\Cl{G}\op\rightarrow\Cl{G}\op/G$$
is evidently constant on orbits, and therefore one obtains a induced map 
\begin{equation}\label{isotropymap/G}
    X/G\xrightarrow{\text{iso}/G}\Cl{G}\op/G.
\end{equation}

\begin{definition}
    Let $G$ be any topological group, and let $X$ be any Hausdorff topological space equipped with a continuous $G$-action. We define the \textit{orbit type stratification}  of $X/G$ to be the continuous map (\ref{isotropymap/G}). For each orbit $(H)\in\Cl{G}/G$, the corresponding stratum will be denoted by $X_{(H)}$. 
\end{definition}

\subsection{Locally smooth actions}

In this subsection, we consider the orbit type stratification associated with a specific family of actions, called \textit{locally smooth actions (with boundary)}. Our main result is that these induce a $C^0$-stratification on the orbit space. This result is not new, as it could be deduced from example from \cite{popper2000compact}. For the reader's convenience, we include a list of the main definitions involved, and a short sketch of proof. We first recall basic facts and definitions about group actions. Standard references for the subject are for example \cite{bredon1972introduction}, \cite{palais1961existence}.

For $n\in\mathbb{N}$, denote by $\rnum^n_+$ the half-space $\rnum_{\geq 0}\times\rnum^{n-1}$ seen as the subspace $\{x_1\geq 0\}\subset\rnum^n$. 

\begin{definition}
    Let $G$ be any topological group. An \textit{orthogonal action} of $G$ on $\rnum^n$ is a continuous action of $G$ on $\rnum^n$ which factors through $O(n)\subset Top(n)$. An orthogonal action of $G$ on $\rnum^n_+$ is an orthogonal action of $G$ on $\rnum^n$ which restricts to the half space $\rnum^n_+$.
\end{definition}

\begin{remark}\label{orthactionfixx1}
    Notice that an orthogonal action on $\rnum^n$ restricts to $\rnum^n_+$ if and only if it fixes the $x_1$-axis. Therefore, such an action restricts to the hyperplane $\{x_1=0\}$.
\end{remark}

\begin{definition}
    Let $X$ be a space with $G$-action, and let $H$ be a closed subgroup of $G$. The \textit{twisted product} $G\times_H X$ is defined to be the orbit space of $G\times X$ under the $H$ action $(g, x, h)\mapsto (gh^{-1}, hx)$.
\end{definition}

\begin{definition}
    Let $G$ be a topological group. A \textit{linear tube} (respectively, a \textit{linear tube with boundary}) is a $G$-space of the form $G\times_H \rnum^n$ (respectively, $G\times_H \rnum^n_+$), for some orthogonal $G$-action on $\rnum^n$ (respectively, on $\rnum^n_+$).
\end{definition}

    \begin{definition}
        Let $G$ be a Lie group, and let $X$ be any Hausdorff topological space. A \textit{locally smooth action with boundary on $X$} is a proper $G$-action on $X$ with the property that, for any orbit $P$ of type $G/H$, there exists a $G$-equivariant open embedding
        $U\hookrightarrow X$
        onto a neighbourhood of $P$, where $U$ is either a linear tube or a linear tube with boundary.
    \end{definition}

    \begin{remark}\label{locsmoothactismanifold}
        One can show that the map $G\times_H X\rightarrow G/H$ is a fiber bundle with fiber $X$ (see \cite[Theorem 2.4]{bredon1972introduction} applied to the $H$-torsor $G\rightarrow G/H$). As a consequence, if $X$ is equipped with a locally smooth $G$-action with boundary, then $X$ is a topological manifold with boundary.
    \end{remark}

    \begin{lemma}
        Let $X$ be a topological space equipped with a locally smooth $G$-action with boundary. Then the $G$-action on $X$ restricts to $\partial X$. In particular, we get a continuous map $X/G\xrightarrow{\partial} [1]$ whose fiber over $0\in[1]$ is $\partial X / G$.
    \end{lemma}

    \begin{proof}
        Since the desired conclusion can be checked locally, the lemma follows from \cref{orthactionfixx1}.
    \end{proof}

    \begin{definition}
        The \textit{refined orbit type stratification} is the continuous map $$X/G\xrightarrow{(\partial, \text{iso}/G)}[1]\times\Cl{G}\op/G.$$
    \end{definition}

    \begin{proposition}\label{reforbtypisC0}
        Let $X$ be a topological space equipped with a locally smooth $G$-action with boundary. Then the refined orbit type stratification on $X/G$ is a $C^0$-stratification.
    \end{proposition}

    \begin{proof}
        Assume that the dimension of $X$ is $n$. We prove the proposition by induction on $n$. 
        
        If $n=0$, one sees that the orbit space $X/G$ is a $0$-dimensional manifold. So assume that $n>0$. Since the statement is local, we need to show that for any linear tube $U\hookrightarrow X$ near an orbit of type $H$, $U/G$ with its own orbit type stratification is $C^0$-stratified. Let $Y$ be either $\rnum^n$ or $\rnum^n_+$. By \cite[Proposition 3.3]{bredon1972introduction}, there's a stratified homeomorphism $Y/H\xrightarrow{\cong}U/G$. We only treat the case $Y=\rnum^n_+$, as the other can be dealt with analogously. By \cref{orthactionfixx1}, the action of $H$ on $Y$ fixes the $\{x_1\geq 0\}$-axis, and restricts to an action on the $\{x_1=0\}$-hyperplane. Therefore, we have $Y/H = \rnum_{\geq 0}\times \rnum^{n-1}/H$. Decompose $\rnum^{n-1}$ as a product $\rnum^l\times\rnum^k\cong\rnum^l\times C(S^{k-1})$, where $\rnum^l\times\{0\}$ corresponds to the plane fixed by $H$. Therefore, we get $Y/H \cong \rnum_{\geq 0}\times\rnum^l\times C(S^{k-1}/H)$. Since the action of $H$ on $S^{k-1}$ is locally smooth (as a restriction of an orthogonal action), by the inductive hypothesis we deduce that $S^{k-1}/H$, with its induced stratification, is compact and $C^0$-stratified. Hence we deduce that $Y/H$ is isomorphic to the basic $\rnum_{\geq 0}\times\rnum^{l}\times C(S^{k-1}/H)\cong \rnum^l\times C(\overline{C(S^{k-1}/H)})$, where $\overline{C(S^{k-1}/H)})$ denotes the closed cone on $S^{k-1}/H$.  
    \end{proof}

    \subsection{Quinn's example}

    We are finally ready to introduce the example of a compact $C^0$-stratified space whose exit-path $\infty$-category is not finite.

    \begin{theorem}\label{examplequinn}
        There exists a compact $C^0$-stratified space $X\rightarrow P$ such that the $\infty$-category $\exit{P}{X}$ is not finite.
    \end{theorem}

    \begin{proof}
        In \cite[Proposition 2.1.4]{quinn1982ends}, Quinn provides an example of a locally smooth $G$-action with boundary on a disc $D$, so that the $D/G$, equipped with the orbit type stratification, has non-vanishing mapping cylinder obstruction. By \cite[Proposition 2.1.3]{quinn1982ends}, this means that the open stratum of $D/G$ does not have the homotopy type of a finite CW-complex. 

        Let $X\rightarrow P$ be $D/G$ equipped with its refined orbit-type stratification. By \cref{reforbtypisC0}, $X\rightarrow P$ is a compact $C^0$-stratified space. Since the open stratum of the orbit type stratification agrees with the open stratum in the refined one, \cite[Proposition 2.1.4]{quinn1982ends} implies that the homotopy type of the open stratum in $X\rightarrow P$ is not finite. Therefore, we may conclude by \cref{restrstratapeter}.
    \end{proof}

    We conclude with the following observation.

    \begin{corollary}
        Let $X\rightarrow P$ be the $C^0$-stratified topological space considered in \cref{examplequinn}. Then $X\rightarrow P$ does not admit any conically smooth structure compatible with its $P$ stratification.
    \end{corollary}

    \begin{proof}
        This follows directly from \cref{examplequinn} and \cref{cpt/finstrata=>cprfinexitinC0/smooth}.
    \end{proof}




\newpage
\nocite{*}
\bibliographystyle{alpha}
\bibliography{finite}
\end{document}